\documentclass[a4paper,12 pt]{amsart}
\usepackage{amsmath,amsfonts,amssymb,latexsym}
\usepackage{hyperref}
\usepackage{graphicx}

\newcommand{\RR}{\mathbb{R}}
\newcommand{\CC}{\mathbb{C}}
\newcommand{\NN}{\mathbb{N}}

\newcommand{\QQ}{\mathbb{Q}}
\newcommand{\OO}{\mathcal{O}}
\newcommand{\Ss}{\mathcal{S}}

\newtheorem{Tw}{Theorem}
\newtheorem{Le}{Lemma}

\newtheorem{Wn}{Corollary}

\theoremstyle{definition}
\newtheorem{Df}{Definition}

\theoremstyle{remark}
\newtheorem{Uw}{Remark}

\begin{document}
\title[The Stokes phenomenon]{The Stokes phenomenon for certain partial differential equations with meromorphic initial data}
\author{S{\l}awomir Michalik}
\address{Faculty of Mathematics and Natural Sciences,
College of Science\\
Cardinal Stefan Wyszy\'nski University\\
W\'oycickiego 1/3,
01-938 Warszawa, Poland}
\email{s.michalik@uksw.edu.pl}
\urladdr{\url{http://www.impan.pl/~slawek}}
\author{Bo\.{z}ena Podhajecka}
\address{Faculty of Mathematics and Natural Sciences,
College of Science\\
Cardinal Stefan Wyszy\'nski University\\
W\'oycickiego 1/3,
01-938 Warszawa, Poland}
\email{bpodhajecka@o2.pl}

\begin{abstract}
We study the Stokes phenomenon for the solutions of the $1$-dimensional complex heat equation and its generalizations with meromorphic initial data.
We use the theory of Borel summability for the description of Stokes lines, anti-Stokes lines, jumps across Stokes lines,
and a maximal family of the solutions.
\end{abstract}

\maketitle

\section{Introduction}
In recent years the methods of Borel summability have been intensively used for the study of the asymptotic behavior of solution to non-Kowalevskian
complex partial differential equations with holomorphic initial data in a complex neighborhood of the origin.
The procedure of summability allows us to get the actual solution from a formal one. 
This actual solution is holomorphic in some sectorial neighborhood of the origin and, moreover,
its asymptotic expansion at the origin is given by the formal solution.

The first significant result was obtained for the $1$-dimensional complex heat equation 
\begin{equation}
 \label{eq:h}
 \begin{cases}
  \partial_t u=\partial_z^2 u\\
  u(0,z)=\varphi(z)\in\OO(D),
 \end{cases}
\end{equation}
where $D$ is a complex neighborhood of the origin, and was presented by D.A.~Lutz, M.~Miyake and R.~Sch\"afke in \cite{L-M-S}.
The authors characterized the Borel summability in a direction $\theta$
of the formal solution of the Cauchy problem (\ref{eq:h}) in terms of the analytic continuation property of the Cauchy datum $\varphi$ with the
appropriate growth condition.

In \cite{Miy}, M.~Miyake gave the characterization of the Borel summability for the solutions of the Cauchy problem
for  non-Kowalevskian complex partial differential equations of the form
\begin{equation}
 \label{eq:ghe}
 \begin{cases}
\partial_{t}^{p}u(t,z)=\partial_{z}^{q}u(t,z),\,\, p,q\in\NN,1\leq p<q \\
u(0,z)=\varphi(z) \in\OO(D)\\
\partial_{t}^{j}u(0,z)=0\ for\,\, j=1,2,\dots,p-1.
 \end{cases}
\end{equation}

Another important achievement was given by K.~Ichinobe in \cite{I2}, who defined an integral representation of the Borel sum of
the divergent solution of (\ref{eq:ghe}).

It is worth pointing out that the theory of summability and multisummability of the formal solutions of PDEs
is intensively developed and many other interesting papers were written by W.~Balser \cite{B1,B5}, W.~Balser and M.~Loday-Richaud \cite{B-L},
W.~Balser and M.~Miyake \cite{B-Mi}, O.~Costin, H.~Park and Y.~Takei \cite{C-P-T},
K.~Ichinobe \cite{I,I3}, G.~{\L}ysik \cite{L3,L4}, S.~Malek \cite{Mal2}, S.~\={O}uchi \cite{O3}, H.~Tahara and H.~Yamazawa \cite{T-Y},
M.~Yoshino \cite{Yo} and others.

In this paper, we consider equations with meromorphic initial data, whose formal solutions are summable in all,  but finitely many singular
directions. Hence we get a family of solutions, which we specify using Stokes lines, anti-Stokes lines and jumps across these Stokes lines.

The main focus is on the heat equation (\ref{eq:h}). First, when the Cauchy datum $\varphi$ has a simple pole at the point $z_0\in\CC\setminus\{0\}$,
using the $1$-sum in a direction $d$ of the formal solution $\hat u(t,z)=\sum_{n=0}^{\infty}\frac{\varphi^{(2n)}(z)t^n}{n!}$ of (\ref{eq:h}),
we describe the Stokes phenomenon and a maximal family of solutions. In this way we obtain the following main result
(for the precise formulation see Theorem \ref{th:main}).
\bigskip\par
{\em
Assume that the Cauchy datum of (\ref{eq:h}) is given by
$\varphi(z)=\frac{a}{z-z_0}+\tilde{\varphi}(z)$ for some $a,z_0 \in\CC\setminus\{0\}$ and
$\tilde{\varphi}(z)\in\OO^{2}(\CC)$. 
Set $\delta:=2\*\arg z_0$, $u_1:= u^{\theta}$ for
$\theta\in(\delta,\delta+2\pi)\mod4\pi$ and $u_2:= u^{\theta}$ for $\theta\in(\delta + 2\pi,\delta+4\pi)\mod4\pi$,
where $u^{\theta}$ is a solution of (\ref{eq:h}) given by (\ref{eq:heat_solution}).

Then $\{u_1,u_2\}$ is a maximal family of solutions of (\ref{eq:h}) on the Riemann surface of the square root function.
Moreover, Stokes lines for a formal solution $\hat{u}$ are the sets $\arg t = \delta$ and
$\arg t = \delta+2\pi$, and anti-Stokes lines for $\hat{u}$ are the sets $\arg t = \delta\pm\frac{\pi}{2}$ and
$\arg t = \delta+2\pi\pm\frac{\pi}{2}$. Jumps across the Stokes lines
are given respectively by
\begin{itemize}
 \item $u^{\delta^+}(t,z)-u^{\delta^-}(t,z)=-i\*\sqrt{\pi/t}\*a\*e^{-\frac{(z_0-z)^2}{4t}}$,
 \item $u^{(\delta + 2\pi)^+}(t,z)-u^{(\delta +2\pi)^-}(t,z)=i\*\sqrt{\pi/t}\*a\*e^{-\frac{(z_0-z)^2}{4t}}$.
\end{itemize}
}
\bigskip\par
Next, we extend this result to the Cauchy datum $\varphi$ given by a meromorphic function with finitely many poles, i.e.
$\varphi(z)=\sum_{l=1}^{n}\sum_{k=1}^{r_l}\frac{a_{lk}}{(z-z_l)^k}+\tilde{\varphi}(z)$ with $z_l\in\CC\setminus\{0\}$,
$\tilde{\varphi}(z)\in\OO^2(\CC)$.

Finally, in a similar way, we describe the Stokes phenomenon and a maximal family of solutions for the Cauchy problem (\ref{eq:ghe}).

The paper is organized as follows. In Section 2, we review some of the standard terminology, facts, and theorems about asymptotic expansions
and $k$-summability. In Section 3, we introduce the notion of Stokes lines, anti-Stokes lines, and jumps for $k$-summable formal power
series. Section 4 provides a detailed exposition of a maximal family of actual solutions of a given non-Kowalevskian
equation. In the next section, we present and prove the main results of this paper adding several corollaries as conclusions from 
Theorem \ref{th:main}. In Section 6, we extend Theorem \ref{th:main} to the solutions of the equation (\ref{eq:ghe}), which is a generalization
of the heat equation (\ref{eq:h}).
The last section contains final remarks.

\section{Asymptotic expansions and summability}
In this section we recall necessary definitions, lemmas and theorems about asymptotic expansions and summability.
For more details we refer the reader to \cite{B2}.
\subsection{Sectors and formal power series}
Let us begin with the following definitions:
\begin{Df}
A \textit{sector $S$ in a direction $d$ with an opening $\alpha$ and radius $R$} in the universal covering space $\tilde{\CC}$ of
$\CC\setminus\{0\}$ is a set of the form:
\begin{displaymath}
S=S(d,\alpha,R)=\{z\in\tilde{\CC}\colon\ z=r\*e^{i\phi},\ r\in(0,R),\ \phi\in(d-\alpha/2,d+\alpha/2)\},
\end{displaymath}
where $d\in\RR$, $\alpha>0$ and $R\in\RR_+$.

If $R=+\infty$, then a sector $S$ is \textit{unbounded}, which will be written as $S=S(d,\alpha)$.
\end{Df}

\begin{Df}
A sector $S^{*}$ is called a \textit{proper subsector} of $S=S(d,\alpha,R)$, if
$\overline{S^{*}}\setminus\{0\} \subseteq S$. The notation $S^{*}\prec S$ will be used.
\end{Df}

\begin{Df}
Let $k>0$ and let $f$ be an analytic function in an unbounded sector $S$  (it will be denoted briefly by  $f\in\OO(S)$). A function $f$ is of
\textit{exponential growth of order at most $k$}, if for every subsector $S^*\prec S$ there exist constants $C_1, C_2>0$ such that
$|f(x)|\le C_1\*e^{C_2|x|^{k}}$ for every $x\in S^*$. The set of all such functions will be denoted by $\OO^{k}(S)$.

Analogously, the set of entire functions of exponential growth of order at most $k$ will be denoted by $\OO^{k}(\CC)$.
\end{Df}

Now consider formal power series $\sum_{n=0}^{\infty} a_{n}\*t^{n}$, where $(a_n)^{\infty}_{n=0}$ is a sequence of complex numbers.
The set of all such formal power series will be denoted by $\CC[[t]]$.
More generally, the set of formal power series of the form $\sum_{n=0}^{\infty} a_{n}(z)t^{n}$, where
$a_{n}(z)\in\OO(D_r)$ for all $n\in\NN_{0}$ and $D_r=\{z\in\CC\colon |z|<r\}$, will be denoted by $\OO(D_r)[[t]]$.

If the radius $r$ is not essential, the set $D_r$ will be denoted briefly by $D$.  

\begin{Df}
Let $s\in\RR$. A formal power series $\sum_{n=0}^{\infty} a_{n}\*t^{n}$ is called a \emph{formal power series of Gevrey order $s$}, if there exist $A, B>0$
such that $|a_n|\le A\*B^n\*(n!)^s$ for every $n\in\NN_0$.
The set of all such formal power series will be denoted by $\CC[[t]]_s$ (resp. $\OO(D_r)[[t]]_s$).
\end{Df}

\subsection{Asymptotic expansions}
In this subsection we restrict our attention to the Gevrey asymptotics.

\begin{Df}
Let $S$ be a given sector in the universal covering space $\tilde{\CC}$ and $f\in\OO(S)$. A formal power series
$\hat{f}(t)=\sum_{n=0}^{\infty} a_{n}\*t^{n}\in\CC[[t]]_{s}$
of Gevrey order $s$ $(s\in\RR)$
 is called \textit{Gevrey's asymptotic expansion of order $s$} of the function $f$ in $S$ if  for every $S^*\prec S$ there exist $A, B>0$ such that for every
 $N\in\NN_0$ and every $t\in S^*$
 $$|f(t)-\sum_{n=0}^{N} a_{n}\*t^{n}|\le A\*B^N\*(N!)^{s}\*|t|^{N+1}.$$
If this is so, one can use the notation $f(t)\sim_s\hat{f}(t)$ in $S$.
\end{Df}

We recall two important theorems about the Gevrey asymptotics.
\begin{Tw}[Ritt's theorem, {\cite[Proposition 10]{B2}}]
Let $\hat{x}(t)\in\CC[[t]]_{s}$, where $s>0$. Let $S$ be a sector of an opening $\alpha$, where $0<\alpha\le s\*\pi$.
Then there exists $x(t)\in\OO(S)$ such that  $x(t)\sim_{s} \hat{x}(t)$ in $S$.
\end{Tw}

\begin{Tw}[Watson's lemma, {\cite[Proposition 11]{B2}}]
\label{th:watson}
Let $S$ be a sector of an opening $\alpha$ such that $\alpha>s\pi$ and $s>0$. Suppose that $x(t)\in\OO(S)$ satisfies $x(t)\sim_{s} 0$ in $S$.
Then $x(t)\equiv 0$ in $S$.
\end{Tw}

\subsection{Summability}
\label{sect}
This subsection presents a method which is usually used for the investigation of the formal solutions of PDEs with constant coefficients.
The title procedure consists in the change of a formal power series $\hat{x}(t)=\sum_{n=0}^{\infty} a_{n}t^{n}\in\CC[[t]]$
into a holomorphic function on some sectorial neighborhood of the origin.

Fix $k>0$, $d\in\RR$ and $\varepsilon>0$. First, it will be used a \textit{Borel modified transform of order $k$} defined by
$$(\breve{B_{k}}\hat{x})(t):=\sum_{n=0}^{\infty}\frac{a_n\,\*t^n\,\*n!}{\Gamma{\bigl(1+n(1+\frac{1}{k})\bigr)}}.$$
If the above power series is convergent then one can apply the \textit{Ecalle operator $\textit{E}_{k,d}$ of order $k$ in the direction $d$}, i.e.
\begin{equation}
\label{eq:ecalle}
(\textit{E}_{k,d}\,g)(t):=t^{-k/(1+k)}\int_{e^{i\theta}\RR_+}g(s)\,C_{(k+1)/k}((s/t)^{\frac{k}{1+k}})\,ds^{\frac{k}{1+k}},
\end{equation}
where $g(s)=(\breve{B_{k}}\hat{x})(s)$, a function $C_{\alpha}$ (see \cite[Section 11]{B2})  is of the form
\begin{gather*}
 C_{\alpha}(\tau):=\sum_{n=0}^{\infty}\frac{(-\tau)^n}{n!\,\*\Gamma\bigl(1-\frac{n+1}{\alpha}\bigr)}\quad\text{for}\quad \alpha>1\,\,\, 
\end{gather*}
and the integration is taken over any ray $e^{i\theta}\RR_+:=\{re^{i\theta}\colon r\geq 0\}$ with $\theta\in (d-\varepsilon/2,d+\varepsilon/2)$.

\bigskip

Now let us check when the integral in (\ref{eq:ecalle}) is convergent. For this purpose we recall an useful lemma about the function $C_{\alpha}$:
\begin{Le}{\cite[Lemma 5.1]{B0}}
\label{le:kk}
Let $\alpha>1$ and $\beta\neq 0$ are such that $\frac{1}{\beta}+\frac{1}{\alpha}=1$.
Then for every $\varepsilon>0$ there exist positive constants $A_1, A_2$ such that $|C_{\alpha}(\tau)|\leq A_1\*e^{-A_{2}|\tau|^{\beta}}$
for every $\tau$ such that $|\mathrm{arg} \tau|\leq \frac{\pi}{2\*\beta}-\varepsilon$.
\end{Le}

\begin{Uw}
\label{re:rem 2}
If $\alpha=(k+1)/k$, then $\beta=k+1$. Thus Lemma \ref{le:kk} implies that for every $\tilde \varepsilon>0$ there exist $A_1,A_2>0$ such that
$$\bigl|C_{(k+1)/k}((s/t)^{\frac{k}{1+k}})\bigr|\leq A_1\*e^{-A_2\*|\frac{s}{t}|^{k}}$$
for every $s/t$ satisfying $|\arg(s/t)|\leq \frac{\pi}{2k}-\frac{\tilde \varepsilon}{2}$. 
\end{Uw}

By the above remark, the integral in (\ref{eq:ecalle}) is convergent if the function $g(s)$
is of exponential growth of order at most $k$.

\bigskip
Summarizing, one can use the modified $k$-summability method in the direction $d$ if $\textit{E}_{k,d}g$ is well defined,
i.e. if the following conditions are satisfied:
\begin{enumerate}
\item[(A)]
 $\hat{x}(t)\in\CC[[t]]_{1/k}$,
\item[(B)] $(\breve{B}_{k}\hat{x})(t)\in\OO^k(S_d)$, where $S_d=S(d,\varepsilon)$ is an unbounded sector in the direction $d$.
\end{enumerate}

Using the general theory of moment summability (see \cite[Section 6.5]{B2}) with a kernel function (see \cite[Section 11.1]{B2})
$$
e(z)=\frac{k}{1+k}z^{\frac{k}{1+k}}C_{(k+1)/k}(z^{\frac{k}{1+k}})
$$
of order $k$ and with a corresponding moment function
$$m(u)=\Gamma(1+u(1+\frac{1}{k}))/\Gamma(1+u),$$
by \cite[Theorem 38]{B2} one can conclude that the modified $k$-summability method is equivalent to the standard $k$-summability method
(see \cite[Section 6.2]{B2}).

For this reason, if the conditions (A) and (B) holds, then one can say that $\hat{x}(t)$ is 
\emph{$k$-summable in the direction $d$}. Moreover, the \emph{$k$-sum of $\hat{x}(t)$ in the direction $d$} is given by
$$x^d(t)=\Ss_{k,d}\hat{x}(t):=(\textit{E}_{k,d}\breve{B_{k}}\hat{x})(t).$$

\begin{Uw} 
\label{re:rem 3}
By Remark \ref{re:rem 2} for every  
$\tilde\varepsilon>0$ there exist $A_1,A_2>0$ such that $|C_{(k+1)/k}((s/t)^{\frac{k}{1+k}})\bigr|\leq
A_1\*e^{-A_2\*|\frac{s}{t}|^{k}}$ for every $s/t$ satisfying $|\arg(s/t)|\leq \frac{\pi}{2k}-\frac{\tilde \varepsilon}{2}$.
If $\arg s=\theta$ then 
$\arg t\in (\theta -\frac{\pi}{2k}+\frac{\tilde \varepsilon}{2}, \theta + \frac{\pi}{2k}-\frac{\tilde \varepsilon}{2})$.
If $g(s)=(\breve{B}_{k}\hat{x})(t)\in\OO^k(S_d)$, i.e. there exist $B_1,B_2>0$ such that $|g(s)|\leq B_1 e^{B_2|s|^k}$,
then
$$\Bigl|\int_{ e^{i\theta}\RR_+}g(s)\,C_{(k+1)/k}((s/t)^{\frac{k}{1+k}})\,ds^{\frac{k}{1+k}}\Bigr|\leq 
\int_0^{\infty}B_1 e^{B_2|s|^k}\cdot
A_1\*e^{-A_2\*|\frac{s}{t}|^{k}}\,ds^{\frac{k}{1+k}}.
$$
The integral above converges if $A_2|t|^{-k}>B_2,$ so $|t|<(\frac{A_2}{B_2})^{1/k}.$
Thus $(\textit{E}_{k,d}\,g)(t)$ is a well-defined holomorphic function for $t\in S(\theta,\frac{\pi}{k}-\tilde{\varepsilon},r)$,
where $r=(\frac{A_2}{B_2})^{1/k}$.
Since $\theta\in (d-\varepsilon/2,d+\varepsilon/2)$,
$$
x^d(t)=\Ss_{k,d}\hat{x}(t):=(\textit{E}_{k,d}\breve{B_{k}}\hat{x})(t)\sim_{1/k}\hat x(t)\quad \textrm{in}\quad
S(d,\frac{\pi}{k}+\varepsilon-\tilde\varepsilon,r).
$$
Thus for every $\hat{\varepsilon}\in (0,\varepsilon)$ there exists $r>0$ such that $x^d(t)\in\OO(S(d,\pi/k+\hat{\varepsilon},r))$.
\end{Uw}

\begin{Uw}
\label{re:unique}
By Watson's lemma (Theorem \ref{th:watson}), $k$-sum $x^d(t)$ is the unique holomorphic function on $S(d,\pi/k +\hat{\varepsilon},r)$ satisfying 
$$x^d(t)\sim_{1/k}\hat{x}(t)\,\,\mathrm{in}\,\,S(d,\pi/k+\hat{\varepsilon},r).$$
\end{Uw}

\begin{Df}
If $\hat{x}$ is $k$-summable in all directions $d$ but (after identification modulo $2\pi$) finitely many directions $d_1,\dots,d_n$ then
$\hat{x}$ is called \emph{$k$-summable} and $d_1,\dots,d_n$ are called \emph{singular directions of $\hat{x}$}.
\end{Df}

\begin{Uw}
 The notion of asymptotic expansion and $k$-summability can be extended in a natural way to the formal power series $\hat{u}(t,z)\in\OO(D_r)[[t]]$.
\end{Uw}

\section{The Stokes phenomenon for summable power series}
The concept of the Stokes phenomenon for $k$-summable formal power series $\hat{x}=\hat{x}(t)\in\CC[[t]]_{1/k}$ (resp.
$\hat{u}=\hat{u}(t,z)\in\OO(D)[[t]]_{1/k}$) is introduced by the following definition.
\begin{Df}
Assume that $\hat{x}\in\CC[[t]]_{1/k}$ (resp. $\hat{u}\in\OO(D)[[t]]_{1/k}$) is $k$-summable in all directions
$d\in (\phi-\varepsilon, \phi+\varepsilon)$,
but a singular direction $d=\phi$ (for some $\varepsilon>0$). A set $\mathcal{L}_{\phi}=\{t\in\tilde{\CC}\colon \arg t=\phi\}$ is called
the \emph{Stokes line for $\hat{x}$} (resp. $\hat{u}$).

Moreover, if $\phi^+$ (resp. $\phi^-$) denotes a direction close to $\phi$ and greater (resp. less) than $\phi$,
and $x^{\phi^+}=\Ss_{k,\phi^+}\hat{x}$ (resp. $x^{\phi^-}=\Ss_{k,\phi^-}\hat{x}$)
then the difference $x^{\phi^+}-x^{\phi^-}$ is called a \emph{jump for $\hat{x}$ across the Stokes line $\mathcal{L}_{\phi}$}. Analogously
we define a \emph{jump for $\hat{u}$ across the Stokes line $\mathcal{L}_{\phi}$}.
\end{Df}

\begin{Uw}
 Every Stokes line $\mathcal{L}_{\phi}$ for $\hat{x}$ (resp. $\hat{u}$) determines also so called \emph{anti-Stokes lines}
 $\mathcal{L}_{\phi\pm\frac{\pi}{2k}}$ for $\hat{x}$ (resp. $\hat{u}$), which are also often investigated.
\end{Uw}

\begin{Uw}
Assume that $S$ is a sector with an opening $\pi/k$ in a direction $\phi$. Let $f(t), g(t)\in\OO(S)$ be $k$-sums of $\hat{x}(t)$
in directions $\phi^-$ and $\phi^+$ respectively. It means that $f(t)\sim_{1/k}\hat{x}(t)$ and $g(t)\sim_{1/k}\hat{x}(t)$ on $S$.
Set $r(t):=|f(t)-g(t)|$ for all $t\in S$. Then $r(t)$ is minimal on the Stokes line $\mathcal{L}_{\phi}$ for $t$ close to zero
and satisfies inequalities: $|f(t)|\leq r(t)$ or $|g(t)|\leq r(t)$ on the anti-Stokes lines $\mathcal{L}_{\phi\pm\frac{\pi}{2k}}$.
\end{Uw}

\section{A maximal family of solutions}
In this section we describe a family of actual solutions of given non-Kowalevskian equation using $k$-sums of $k$-summable formal
power series solution. More precisely we consider the Cauchy problem
\begin{equation}
 \label{eq:pde}
 \begin{cases}
  P(\partial_t,\partial_z)u=0&\\
  \partial_t^j u(0,z)=\varphi_j(z)\in\OO(D),&j=0,\dots,m-1,
 \end{cases}
\end{equation}
where $P(\partial_t,\partial_z)=\partial_t^m-\sum_{j=1}^m\partial_t^{m-j}P_j(\partial_z)$ is a differential operator of order $m$
with respect to $\partial_t$. Since the principal part of the operator $P(\partial_t,\partial_z)$ with respect to $\partial_t$ is given by $\partial_t^m$,
the Cauchy problem (\ref{eq:pde}) has a unique formal power series solution $\hat{u}\in\OO(D)[[t]]$.
If we additionally assume that $\hat{u}$ is $k$-summable, then using the procedure of $k$-summability in nonsingular directions,
we obtain a family of actual solutions of (\ref{eq:pde})
on some sectors with opening greater than $\pi/k$ with respect to $t$. This motivates us to introduce the following definitions.

\begin{Df}
Let $S$ be a sector in the universal covering space $\tilde{\CC}$. A function $u\in\OO(S\times D)$ is called an \emph{actual solution}
of (\ref{eq:pde})
if it satisfies 
\begin{equation*}
 \begin{cases}
  P(\partial_t,\partial_z)u=0&\\
  \lim\limits_{t\to 0,\ t\in S}\partial_t^j u(t,z)=\varphi_j(z)\in\OO(D),&j=0,\dots,m-1.
 \end{cases}
\end{equation*}
\end{Df}

\begin{Df}
 \label{df:maximal}
 Assume that a unique formal power series solution $\hat{u}$ of (\ref{eq:pde}) is $k$-summable, $I$ is a finite set of indices, and $V$
 is a sector with an opening greater than $\pi/k$ on the Riemann surface of $t^{\frac{1}{q}}$ for some $q\in\QQ_+$.
 We say that $\{u_i\}_{i\in I}$ with $u_i\in\OO(V_i\times D)$ is a
 \emph{maximal family of solutions of (\ref{eq:pde}) on $V\times D$} if the following conditions hold:
 \begin{itemize}
  \item $V_i\subseteq V$ is a sector with an opening greater than $\pi/k$ for every $i\in I$.
  \item $\{V_i\}_{i\in I}$ is a covering of $V$. 
  \item $u_i\in\OO(V_i\times D)$ is an actual solution of (\ref{eq:pde}) for every $i\in I$.
  \item If $V_i\cap V_j \neq \emptyset$ then $u_i\not\equiv u_j$ on $(V_i\cap V_j)\times D$ for every $i,j\in I$, $i\neq j$.
  \item If $\tilde{u}\in\OO(\tilde{V}\times D)$ is an actual solution of (\ref{eq:pde}) for some sector $\tilde{V}\subseteq V$
  and $\tilde{u}=\Ss_{k,d}\hat{u}$ for some nonsingular direction $d$, then there exists $i\in I$ such that
  $\tilde{V}\cap V_i\neq\emptyset$ and $\tilde{u}\equiv u_i$
  on $(\tilde{V}\cap V_i)\times D$.
 \end{itemize}
\end{Df}

Now we are ready to describe a maximal family of solutions of (\ref{eq:pde})
\begin{Tw}
 \label{th:general}
 Let $\hat{u}$ be a $k$-summable formal power series solution of (\ref{eq:pde}) with a $k$-sum in a nonsingular direction $d$ given by
 $u^d=\Ss_{k,d}\hat{u}$.
 Assume that there exists $q\in\QQ_+$, which is the smallest positive rational number such that $u^d(t,z)=u^d(te^{2q\pi i},z)$ for every nonsingular direction $d$.
 Suppose that the set of singular directions of $\hat{u}$ modulo $2q\pi$ is given by $\{d_1,\dots,d_n\}$,
 where $0\leq d_1 < \dots< d_n<2q\pi$.
 Then Stokes lines for $\hat{u}$ are the sets $\mathcal{L}_{d_i}$ and anti-Stokes lines are the sets $\mathcal{L}_{d_i\pm\frac{\pi}{2k}}$
 for $i=1,\dots,n$.
 
 If additionally for every $i=1,\dots,n$ set $u_i:=\Ss_{k,\theta}\hat{u}$ with $\theta\in (d_{i},d_{i+1})$, where
 $d_{n+1}:=d_1+2q\pi$,
 then for every sufficiently small $\varepsilon>0$ there exists $r>0$ such that $u_i\in\OO(V_i(\varepsilon,r)\times D)$, where
 $$V_i(\varepsilon,r):=\{t\in W\colon |t|\in(0,r), \arg t \in (d_{i}-\frac{\pi}{2k}+\frac{\varepsilon}{2},d_{i+1}+\frac{\pi}{2k}-\frac{\varepsilon}{2}) \mod 2q\pi\}$$
 and $W$ is the Riemann surface of $t^{\frac{1}{q}}$.
 Moreover, $\{u_1,\dots,u_n\}$ with $u_i\in\OO(V_i(\varepsilon,r)\times D)$ is a maximal family of solutions of (\ref{eq:pde}) on $W_r\times D$, where
 $W_r=\{t\in W\colon 0<|t|<r\}$.
\end{Tw}
\begin{proof}
 Immediately by the definition we conclude that $\mathcal{L}_{d_i}$ is a Stokes line and $\mathcal{L}_{d_i\pm\frac{\pi}{2k}}$ are 
 anti-Stokes lines. They play an important
 role in our description of the maximal family of solutions of (\ref{eq:pde}). To this end observe that for every sufficiently small
 $\varepsilon>0$ there exists $r>0$ such that
 $\Ss_{k,\theta}\hat{u}\in\OO(S(\theta,\frac{\pi}{k}-\varepsilon,r)\times D)$. Moreover, since by \cite[Lemma 10]{B2} 
 \begin{gather*}
  \Ss_{k,\theta_1}\hat{u}=\Ss_{k,\theta_2}\hat{u}\quad \textrm{for every} \quad \theta_1,\theta_2\in (d_{i},d_{i+1}),
 \end{gather*}
 the function $u_i$ is well defined and is
 analytically continued to the set $V_i(\varepsilon,r)\times D$.
 
 We will show that $\{u_1,\dots,u_n\}$ is a maximal family of solutions of (\ref{eq:pde}) on $W_r\times D$. First,
 we take such small $\varepsilon > 0$ that the opening of $V_i(\varepsilon,r)$ ($V_i$ for short) is greater than $\frac{\pi}{k}$
 for every $i=1,\dots,n$.
 Of course
 $\{V_i\}_{i=1,\dots,n}$ is a covering of $W_r$.
 
 Remark \ref{re:unique} implies that $u_i(t,z)\sim_{1/k} \hat{u}(t,z)$ in $V_i$.
 Hence by the Gevrey asymptotics properties \cite[Theorems 18--20]{B2} 
 \begin{gather*}
  P(\partial_t,\partial_z)u_i(t,z)\sim_{1/k} P(\partial_t,\partial_z)\hat{u}(t,z)=0\quad \textrm{in}\quad V_i.
 \end{gather*}
 Since the opening  of $V_i$ is greater than $\frac{\pi}{k}$, by Remark \ref{re:unique} we conclude that 
 $P(\partial_t,\partial_z)u_i=0$.
 Additionally, since 
 \begin{gather*}
  u_i(t,z)\sim_{1/k} \hat{u}(t,z)=\sum_{j=0}^{\infty}u_j(z)t^j\quad \textrm{in}\ V_i,
 \end{gather*}
 by \cite[Proposition 8]{B2} we get
 \begin{gather*}
   \lim\limits_{t\to 0, t\in V_i}\partial_t^ju_i(t,z)=j!u_j(z)=\varphi_j(z)\quad \textrm{for}\quad j=0,\dots,m-1.
 \end{gather*}
 Therefore $u_i $ is an actual solution of (\ref{eq:pde}) for $i=1,\dots,n$.
 
 Now, let us assume that $V_i\cap V_j\neq \emptyset$ and $u_i\equiv u_j$ on $(V_i\cap V_j)\times D$ for some $i,j\in\{1,\dots,n\}$ and $i\neq j$.
 It means that
 $u_i$ and $u_j$ are analytically continued to $\bar{u}\in\OO((V_i\cup V_j)\times D)$.
 We may assume that 
 $$V_i\cup V_j=\{t\in W\colon |t|\in(0,r), \arg t \in (d_{i}-\frac{\pi}{2k}+\frac{\varepsilon}{2},d_{j+1}+\frac{\pi}{2k}-\frac{\varepsilon}{2})\mod
 2q\pi\}.$$
 Hence, in particular $V_i,V_{i+1}\subset V_i\cup V_j$. Therefore, by Remark \ref{re:unique} we conclude that $\Ss_{k,d_{i+1}^-}\hat{u}=
 \tilde{u}=\Ss_{k,d_{i+1}^+}\hat{u}$. However, by \cite[Proposition 12]{B2} it means that also $\hat{u}$ is $k$-summable in a direction $d_{i+1}$,
 but $d_{i+1}$ is a singular direction.
 So, if $V_i\cap V_j\neq \emptyset$ then $u_i\not\equiv u_j$ on $(V_i\cap V_j)\times D$ for every $i,j\in\{1,\dots,n\}$, $i\neq j$.
 
 By the construction of the family $\{u_1,\dots,u_n\}$, the last condition in Definition \ref{df:maximal} is also satisfied,
 which completes the proof.
\end{proof}

\begin{Uw}
 One can treat a maximal family of solutions $\{u_i\}_{i\in I}$, $u_i\in\OO(V_i\times D)$ as a $k$-precise quasi-function $(\{u_i\},\{V_i\})$ introduced
 by J.-P. Ramis \cite{R}, i.e. as the analytic function which is defined modulo functions with exponential decrease of order $k$.
\end{Uw}

\section{The heat equation}
In this section we consider the initial value problem for the heat equation
\begin{equation}
 \label{eq:heat}
 \begin{cases}
  \partial_t u=\partial_z^2 u\\
  u(0,z)=\varphi(z).
 \end{cases}
\end{equation}
As easily seen, if $\varphi\in\OO(D)$ then this Cauchy problem has a unique formal solution
(which is in general divergent) $$\hat{u}(t,z)=\sum_{n=0}^{\infty}\frac{\varphi^{(2n)}(z)\,\* t^n}{n!}.$$
We use the following result about summable solutions of the heat equation and theirs integral representations
\begin{Tw}[{\cite[Theorem 3.1]{L-M-S} and \cite[Theorem 4.2]{Mic9}}]
\label{th:heat}
Let $d\in\RR$. Suppose that $\hat{u}$ is a unique formal solution of the Cauchy problem of the heat equation (\ref{eq:heat}) with
\begin{gather}
\label{eq:heat_cond}
\varphi\in\OO^2 \biggl(D\cup S(\frac{d}{2},\frac{\varepsilon}{2})\cup S(\frac{d}{2}+\pi,\frac{\varepsilon}{2})\biggr)\quad
\textrm{for\,\, some}\quad \varepsilon>0.
\end{gather}
Then $\hat{u}$ is $1$-summable in the direction $d$ and for every $\theta\in(d-\frac{{\varepsilon}}{2},d+\frac{{\varepsilon}}{2})$ and
for every $\tilde{\varepsilon}\in(0,\varepsilon)$ there exists
$r>0$ such that  its
$1$-sum $u^{\theta}\in\OO(S(\theta,\pi-\tilde{\varepsilon},r)\times D)$ is represented by
\begin{equation}
 \label{eq:heat_solution}
u(t,z)=u^{\theta}(t,z)=
\frac{1}{\sqrt{4\*\pi\*t}}\*\int_{e^{i\*\frac{\theta}{2}}\*\RR_+}\,\,\bigl(\varphi(z+s)+\varphi(z-s)\bigr)\,\*e^{\frac{-s^2}{4t}}ds
\end{equation}
for $t\in S(\theta, \pi-\tilde{\varepsilon},r)$ and $z\in D$.

Moreover, for every $\bar{\varepsilon}\in(0,\varepsilon)$ there exists $r>0$ such that $u\in\OO(S(d,\pi+\bar{\varepsilon},r)\times D)$.
\end{Tw}

Now we describe the Stokes phenomenon and the maximal family of solutions in the case when the Cauchy datum $\varphi$ has a simple pole
at the point $z_0\in\CC\setminus\{0\}$.
\begin{Tw}
\label{th:main}
Assume that the Cauchy datum of (\ref{eq:heat}) is given by
\begin{gather}
\label{eq:heat_data}
\varphi(z)=\frac{a}{z-z_0}+\tilde{\varphi}(z)\quad \text{for some} \quad a,z_0 \in\CC\setminus\{0\} \quad \text{and}\quad
\tilde{\varphi}(z)\in\OO^{2}(\CC). 
\end{gather}
Set $\delta:=2\*\arg z_0$, $u_1:= u^{\theta}$ for
$\theta\in(\delta,\delta+2\pi)\mod4\pi$ and $u_2:= u^{\theta}$ for $\theta\in(\delta + 2\pi,\delta+4\pi)\mod4\pi$,
where $u^{\theta}$ is a solution of (\ref{eq:heat}) given by (\ref{eq:heat_solution}). Finally for $r>0$ denote $W_r=
\{t\in W\colon 0<|t|<r\}$, where $W$ is the Riemann surface of the square root function.

Then for every $\varepsilon>0$ there exists $r>0$ such that $u_l\in\OO(V_l\times D)$ ($l=1,2$), where $V_1=S(\delta+\pi, 3\pi-\varepsilon,r)$, 
$V_2=S(\delta+3\pi, 3\pi-\varepsilon,r)$, and $\{u_1,u_2\}$ is a maximal family of solutions of (\ref{eq:heat}) on $W_r\times D$.
 
Moreover, if $\hat{u}$ is a formal solution of (\ref{eq:heat}) then Stokes lines for $\hat{u}$ are the sets $\mathcal{L}_{\delta}$ and
$\mathcal{L}_{\delta+2\pi}$, and anti-Stokes lines for $\hat{u}$ are the sets $\mathcal{L}_{\delta\pm\frac{\pi}{2}}$ and
$\mathcal{L}_{\delta+2\pi\pm\frac{\pi}{2}}$. Jumps across the Stokes lines $\mathcal{L}_{\delta}$ and $\mathcal{L}_{\delta+2\pi}$
are given respectively by
\begin{itemize}
 \item $u_1(t,z)-u_2(t,z)=u^{\delta^+}(t,z)-u^{\delta^-}(t,z)=-i\*\sqrt{\pi/t}\*a\*e^{-\frac{(z_0-z)^2}{4t}}$\\
 for $(t,z)\in S(\delta,\pi-\varepsilon,r)\times D$,
 \item $u_2(t,z)-u_1(t,z)=u^{(\delta + 2\pi)^+}(t,z)-u^{(\delta +2\pi)^-}(t,z)\\
 =i\*\sqrt{\pi/t}\*a\*e^{-\frac{(z_0-z)^2}{4t}}$
for $(t,z)\in S(\delta+2\pi,\pi-\varepsilon,r)\times D$.
\end{itemize}
\end{Tw}

\begin{Uw}
 \label{re:minus}
 Without loss of generality we may assume in Theorem \ref{th:main} that $\arg z_0\in[0,\pi)$. The other case $\arg(-z_0)\in[0,\pi)$ is reduced to the previous one
 by the substitution $\tilde{u}(t,z):=-u(t,-z)$ or, equivalently, by the replacement of $a$ by $-a$ in (\ref{eq:heat_data}).
\end{Uw}

\begin{proof}[Proof of Theorem \ref{th:main}]
Let $\hat{u}$ be a formal solution of (\ref{eq:heat}) with the Cauchy datum $\varphi$ given by (\ref{eq:heat_data}).
Observe that, if $d\neq\delta\mod 2\pi$ then $\varphi$ satisfies assumption (\ref{eq:heat_cond}).

Hence by Theorem \ref{th:heat}, $\hat{u}$ is $1$-summable in a direction $\theta\in\RR$, $\theta\neq \delta\mod 2\pi$ and
for every $\varepsilon>0$ there exists $r>0$ such that its $1$-sum satisfies
$$u^{\theta}(t,z)=\frac{1}{\sqrt{4\pi\*t}}\*\int_{e^{\frac{i\theta}{2}}\RR_+}\bigl(\varphi(z+\tilde{s})+\varphi(z-\tilde{s})\bigr)
\*e^{-{\tilde{s}}^2/4t}\,d\tilde{s}\,\,
\mathrm{for}\ t\in S(\theta,\pi-\varepsilon,r).$$
Observe that $q=2$ is the smallest positive rational number for which $u^{\theta}(t,z)=u^{\theta}(te^{2q\pi i},z)$.
Moreover, the set of singular directions of $\hat{u}(t,z)$ modulo $4\pi$ is given by $\delta$ and $\delta+2\pi$. Hence by Theorem \ref{th:general},
$\{u_1,u_2\}$ with $u_l\in\OO(V_l\times D)$ ($l=1,2$) is a maximal family of solutions of (\ref{eq:heat}). Moreover,
Stokes lines for $\hat{u}$ are the sets $\mathcal{L}_{\delta}$ and
$\mathcal{L}_{\delta+2\pi}$, and anti-Stokes lines for $\hat{u}$ are the sets $\mathcal{L}_{\delta\pm\frac{\pi}{2}}$.

To calculate jumps across the Stokes lines, take $\theta_1,\theta_2 \neq \delta\mod 2\pi$ and suppose $0<\theta_2-\theta_1<\pi$.
We will show that
\begin{equation}
 \label{eq:difference}
 u^{\theta_2}(t,z)-u^{\theta_1}(t,z)=-2\pi\*i\*\sum_{\tilde z\in B(G,\varphi,z)}\mathrm{res}_{\tilde s=\tilde z}\,
 \Bigl[\frac{1}{\sqrt{4\pi\*t}}\bigl(\varphi(z+\tilde{s})+\varphi(z-\tilde{s})\bigr)\*e^{-{\tilde{s}}^2/4t}\Bigr]
\end{equation}
for $t\in S(\theta_1,\pi-\varepsilon,r)\cap S(\theta_2,\pi-\varepsilon,r)$ and $z\in D$,
 where $G=G(\frac{\theta_1}{2},\frac{\theta_2}{2})=\{z\in\tilde{\CC}\colon \arg z\in(\frac{\theta_1}{2},\frac{\theta_2}{2})\}$ and 
 $B(G,\varphi,z)$ is the set of all points in the sector $G$
at which the function $\tilde{s}\mapsto\frac{1}{\sqrt{4\pi\*t}}\bigl(\varphi(z+\tilde{s})+\varphi(z-\tilde{s})\bigr)\*e^{-{\tilde{s}}^2/4t}$
has poles.\\

To this end for $r>|z_0-z|$ set  $C_r=C_{1,r}\cup C_{2,r}\cup C_{3,r}$,\,\,  where $C_{1,r}=\{tr\*e^{\frac{i\*\theta_2}{2}}:t\in[0,1]\}$,
$C_{2,r}=\{r\*e^{\frac{i(t\theta_1+(1-t)\theta_2)}{2}}:t\in[0,1]\}$ and $C_{3,r}=\{(1-t)r\*e^{\frac{i\*\theta_1}{2}}:t\in[0,1]\}$.
The contour $C_r$ is shown in Figure \ref{fig:1}.

\begin{figure}[h]
\centering
 \includegraphics[width=.7\textwidth]{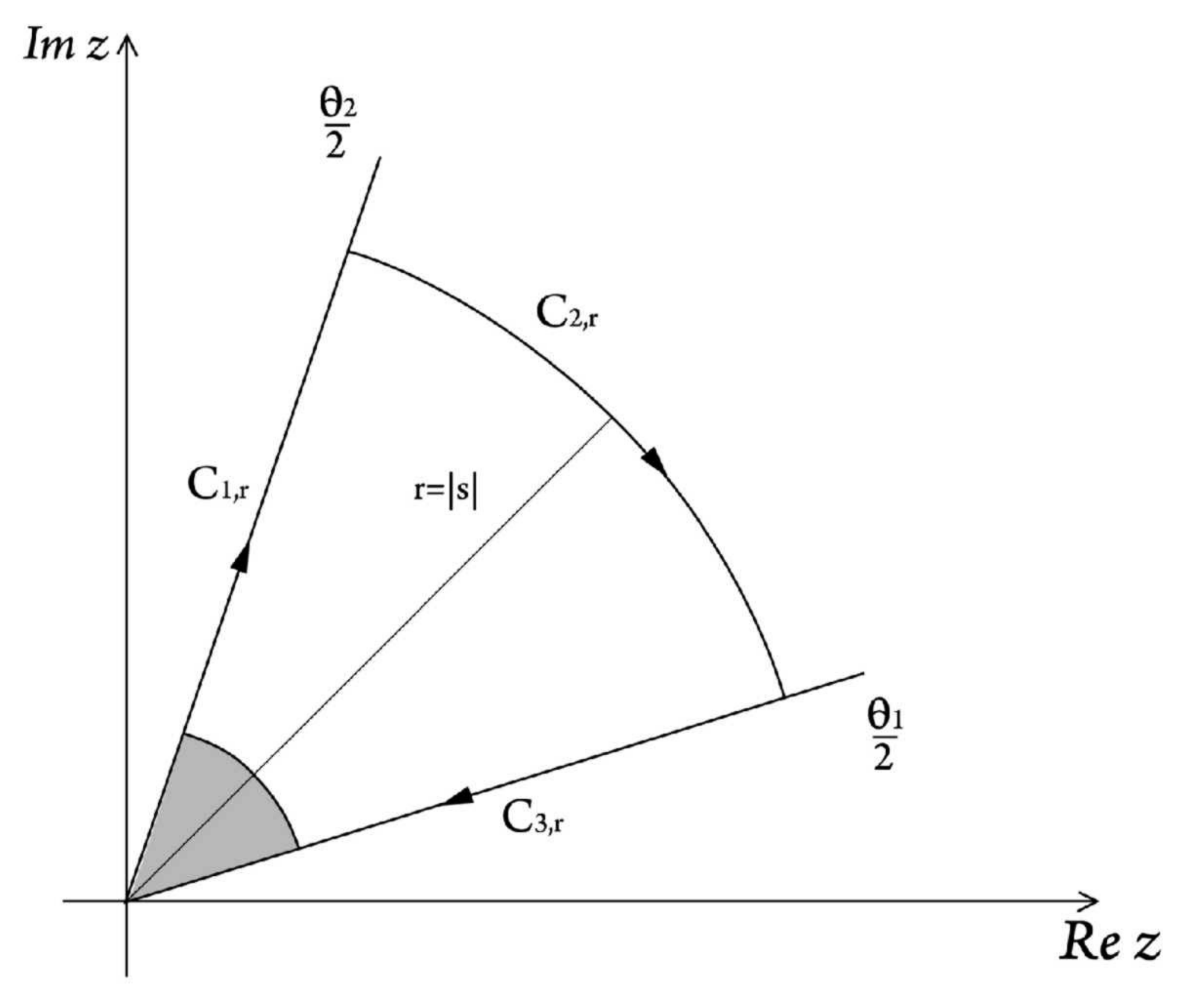}
\caption{The contour $C_r$ in the complex plane.}
\label{fig:1}
\end{figure}

Since $\varphi\in\OO^2(\CC)$, we get
\begin{multline*}
u^{\theta_2}(t,z)-u^{\theta_1}(t,z)
=\lim_{r\longrightarrow\infty}\Biggl[\int_{C_{1,r}}\frac{1}{\sqrt{4\pi\*t}}\bigl(\varphi(z+\tilde{s})+\varphi(z-\tilde{s})\bigr)
\*e^{-{\tilde{s}}^2/4t}\,d\tilde{s}\\
+\int_{C_{3,r}}\frac{1}{\sqrt{4\pi\*t}}\bigl(\varphi(z+\tilde{s})+\varphi(z-\tilde{s})\bigr)\*e^{-{\tilde{s}}^2/4t}\,d\tilde{s}\Biggr],
\end{multline*}
and
\begin{multline*}
\int_{C_r}\frac{1}{\sqrt{4\pi\*t}}\bigl(\varphi(z+\tilde{s})+\varphi(z-\tilde{s})\bigr)\*e^{-{\tilde{s}}^2/4t}\,d\tilde{s}\\
=\int_{C_{1,r}}\frac{1}{\sqrt{4\pi\*t}}\bigl(\varphi(z+\tilde{s})+\varphi(z-\tilde{s})\bigr)\*e^{-{\tilde{s}}^2/4t}\,d\tilde{s}\\ 
 + \int_{C_{2,r}}\frac{1}{\sqrt{4\pi\*t}}\bigl(\varphi(z+\tilde{s})+\varphi(z-\tilde{s})\bigr)\*e^{-{\tilde{s}}^2/4t}\,d\tilde{s}\\ 
+\int_{C_{3,r}}\frac{1}{\sqrt{4\pi\*t}}\bigl(\varphi(z+\tilde{s})+\varphi(z-\tilde{s})\bigr)\*e^{-{\tilde{s}}^2/4t}\,d\tilde{s},
\end{multline*}
also
\begin{multline*}
\Bigg|\int_{C_{2,r}}\frac{1}{\sqrt{4\pi\*t}}\bigl(\varphi(z+\tilde{s})+\varphi(z-\tilde{s})\bigr)\*e^{-{\tilde{s}}^2/4t}\,d\tilde{s}\Bigg| \\
\leq\, \frac{\theta_2-\theta_1}{4\sqrt{\pi\*t}}
\,\mathrm{sup}_{|\tilde{s}|=r}\Bigg|(\varphi(z+\tilde{s})+\varphi(z-\tilde{s})\bigr)\*e^{-{\tilde{s}}^2/4t}\Bigg|\,\*r\\
\leq A\*e^{-Br^2}\,\*r \longrightarrow 0\quad \mathrm{as}\quad r\longrightarrow\infty\quad (\textrm{for some}\ A,B>0),
\end{multline*}

By the residue theorem we obtain
\begin{multline*}
\int_{C_r}\frac{1}{\sqrt{4\pi\*t}}\bigl(\varphi(z+\tilde{s})+\varphi(z-\tilde{s})\bigr)\*e^{-{\tilde{s}}^2/4t}\,d\tilde{s}\\
=-2\pi\*i\*\sum_{\tilde z\in B(G_r,\varphi,z)}\mathrm{res}_{\tilde s=\tilde z}\,\,
\Bigl[\frac{1}{\sqrt{4\pi\*t}}\bigl(\varphi(z+\tilde{s})+\varphi(z-\tilde{s})\bigr)\*e^{-{\tilde{s}}^2/4t}\Bigr],
\end{multline*}
where $G_r=\{z\in\tilde{\CC}\colon |z|<r,\ \arg z\in(\frac{\theta_1}{2},\frac{\theta_2}{2})\}$.

Hence
\begin{multline*}
u^{\theta_2}(t,z)-u^{\theta_1}(t,z)=\lim_{r\longrightarrow\infty}\Biggl[\int_{C_{1,r}}\frac{1}{\sqrt{4\pi\*t}}\bigl(\varphi(z+\tilde{s})+
\varphi(z-\tilde{s})\bigr)\*e^{-{\tilde{s}}^2/4t}\,d\tilde{s}\\
+\int_{C_{3,r}}\frac{1}{\sqrt{4\pi\*t}}\bigl(\varphi(z+\tilde{s})+\varphi(z-\tilde{s})\bigr)\*e^{-{\tilde{s}}^2/4t}\,d\tilde{s}\Biggr]\\
=\lim_{r\longrightarrow\infty}\Bigg[\int_{C_r}\frac{1}{\sqrt{4\pi\*t}}
\bigl(\varphi(z+\tilde{s})+\varphi(z-\tilde{s})\bigr)\*e^{-{\tilde{s}}^2/4t}\,d\tilde{s}\\
-\int_{C_{2,r}}\frac{1}{\sqrt{4\pi\*t}}\bigl(\varphi(z+\tilde{s})+\varphi(z-\tilde{s})\bigr)\*e^{-{\tilde{s}}^2/4t}\,d\tilde{s}\Bigg]\\
=-2\pi\*i\*\sum_{\tilde z\in B(G,\varphi,z)}\mathrm{res}_{\tilde s=\tilde z}\,
\Bigl[\frac{1}{\sqrt{4\pi\*t}}\bigl(\varphi(z+\tilde{s})+\varphi(z-\tilde{s})\bigr)\*e^{-{\tilde{s}}^2/4t}\Bigr],
\end{multline*}
which gives (\ref{eq:difference}).

Now we consider the following cases:
\bigskip\par
Case 1: Let $\theta_1,\theta_2\in (\delta,\delta+2\pi)$.
Without loss of generality we may assume that $\theta_2>\theta_1$. 
Additionally we also assume that $\theta_2-\theta_1<\pi.$
Since $\theta_1,\theta_2\in(\delta,\delta+2\pi)$, there exists $ \tilde{r}>0$ such that  $ \{B(z_0,\tilde{r})\cup B(-z_0, \tilde{r})\}\cap G=
\emptyset$, where $B(a,r)=\{z\in\CC: |z-a|<r\}$. Then we see that for $z\in D_{\tilde{r}}$ the set $B(G,\varphi,z)$ is empty,
and hence by (\ref{eq:difference})
$$u^{\theta_2}(t,z)=u^{\theta_1}(t,z).$$ 
In the general case we may take an auxiliary direction $\overline{\theta}\in(\theta_1,\theta_2)$ such that $\theta_2-\overline{\theta}<\pi$
and $\overline{\theta}-\theta_1<\pi$. Repeating the previous considerations we also conclude that
$$u^{\theta_2}(t,z)=u^{\overline{\theta}}(t,z)=u^{\theta_1}(t,z).$$

We obtain the same result for $\theta_1,\theta_2\in (\delta+2\pi,\delta+4\pi)$.
Hence the functions $u_1:=u^{\theta}$ for $\theta\in(\delta,\delta+2\pi)$ and $u_2:=u^{\theta}$ for
$\theta\in(\delta+2\pi,\delta+4\pi)$ are well defined, i.e they do not depend on $\theta$.
\begin{figure}[h]
\centering
 \includegraphics[width=\textwidth]{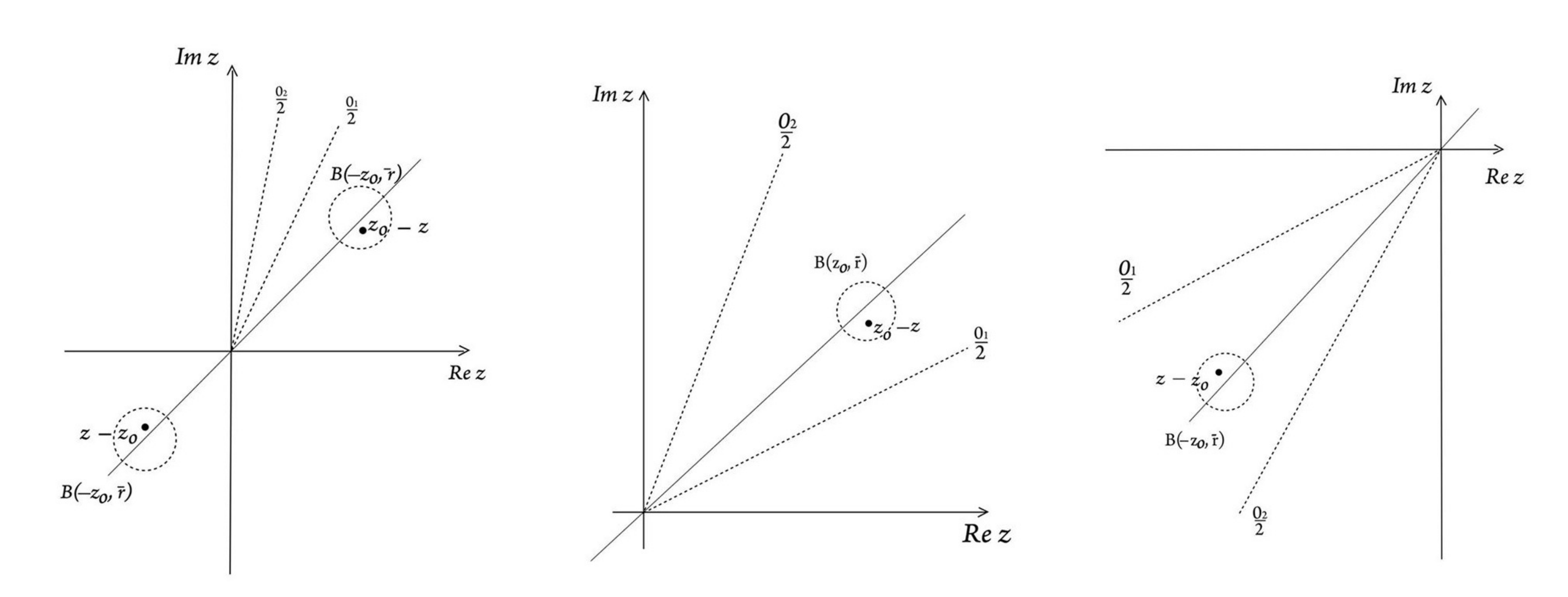}
\caption{From the left: Case 1, Case 2, Case 3.}
\end{figure}
\bigskip\par
Case 2: Let $\theta_1<\delta<\theta_2$. Then there exists $\tilde r>0$ such that $B(z_0,\tilde r)\subset G$. 

Hence we get ${B(G,\varphi,z)=\{z_0-z\}}$ for $z\in D_{\tilde{r}}$. 
Consequently, by (\ref{eq:difference}), we obtain
\begin{multline*}
u^{\theta_2}(t,z)-u^{\theta_1}(t,z)=
-2\pi\*i\,\*\mathrm{res}_{\tilde{s}=z_0-z}\,\Bigl[\frac{1}{\sqrt{4\pi\*t}}\bigl(\varphi(z+\tilde{s})+\varphi(z-\tilde{s})\bigr)
\*e^{-{\tilde{s}}^2/4t}\Bigr]\\
=-i\*\sqrt{\pi/t}\lim_{\tilde{s}\rightarrow z_0-z}(\tilde{s}-(z_0-z))\Bigl[ \bigl(\frac{a}{\tilde{s}-(z_0-z)}+ 
\tilde{\varphi}(z+\tilde{s})+ \frac{a}{z-\tilde{s}-z_0}+\tilde{\varphi}(z-\tilde{s})\bigr)\*e^{-\frac{\tilde{s}^2}{4t}}\Bigr]\\
=-i\*\sqrt{\pi/t}\,\*a\,\*e^{-\frac{(z_0-z)^2}{4t}}
\end{multline*}
for $t\in S(\theta_1,\pi-\varepsilon,r)\cap S(\theta_2,\pi-\varepsilon,r)$ and $z\in D_{\tilde{r}}$.

If we put $\theta_1=\delta^-$
and $\theta_2=\delta^+$, then $(t,z)\in S(\delta,\pi-\varepsilon,r)\times D_{\tilde{r}}$ for small $\varepsilon>0$.
We get 
\begin{gather*}
u^{\delta^+}(t,z)-u^{\delta^-}(t,z)=-i\*\sqrt{\pi/t}\,\*a\,\*e^{-\frac{(z_0-z)^2}{4t}}\quad
\textrm{for}\quad (t,z)\in S(\delta,\pi-\varepsilon,r)\times D_{\tilde{r}}.
\end{gather*}

Notice that for such $(t,z)$ we have $-i\*\sqrt{\pi/t}\,\*a\*\,e^{-\frac{(z_0-z)^2}{4t}}\sim_1 0$,
therefore  $u^{\delta^+}(t,z)\sim_1\hat{u}(t,z)$ and $u^{\delta^-}(t,z)\sim_1\hat{u}(t,z)$.
\bigskip\par
Case 3. Assume that $\theta_1<\delta+2\pi<\theta_2.$ Then $B(-z_0,\tilde r)\subset G$ for some $\tilde{r}>0$.

Analogously to Case 2, for $z\in D_{\tilde{r}}$ we get
$B(G,\varphi,z)=\{z-z_0\}$ and by (\ref{eq:difference}) we derive
\begin{multline*}
u^{\theta_2}(t,z)-u^{\theta_1}(t,z)
=-2\pi\*i\,\*\mathrm{res}_{\tilde{s}=z-z_0}\,\Bigl[\frac{1}{\sqrt{4\pi\*t}}\bigl(\varphi(z+\tilde{s})+\varphi(z-\tilde{s})\bigr)
\*e^{-{\tilde{s}}^2/4t}\Bigr]\\
=-i\*\sqrt{\pi/t}\lim_{\tilde{s}\rightarrow z-z_0}(\tilde{s}-(z-z_0))\Bigl[ \bigl(\frac{a}{\tilde{s}-(z_0-z)}+ \tilde{\varphi}(z+\tilde{s})+
\frac{-a}{\tilde{s}-(z-z_0)}+\tilde{\varphi}(z-\tilde{s})\bigr)\*e^{-\frac{\tilde{s}^2}{4t}}\Bigr]\\
=i\*\sqrt{\pi/t}\,\*a\,\*e^{-\frac{(z_0-z)^2}{4t}}
\end{multline*}
for $t\in S(\theta_1,\pi-\varepsilon,r)\cap S(\theta_2,\pi-\varepsilon,r)$ and $z\in D_{\tilde{r}}$.

If we put $\theta_1={(\delta+2\pi)}^-$ and $\theta_2={(\delta+2\pi)}^+$, then $(t,z)\in S(\delta+2\pi,\pi-\varepsilon,r)\times
D_{\tilde{r}}$ for small $\varepsilon>0$.
We get $$u^{(\delta+2\pi)^+}(t,z)-u^{(\delta+2\pi)^-}(t,z)=i\*\sqrt{\pi/t}\,\*a\,\*e^{-\frac{(z_0-z)^2}{4t}}$$ for
$(t,z)\in S(\delta+2\pi,\pi-\varepsilon,r)\times D_{\tilde{r}}$.
\end{proof}

\begin{Uw}
 The Cauchy problem (\ref{eq:heat}) with a meromorphic Cauchy datum was previously also studied by Lutz, Miyake and Sch\"afke 
 \cite[Section 5]{L-M-S},
 but only in the specific case $\varphi(z)=\frac{1}{z}$, which is not included in our considerations. 
\end{Uw}

Since the equation (\ref{eq:heat}) is linear we get
\begin{Wn}
\label{co:merom}
Assume that the Cauchy datum of (\ref{eq:heat}) is given by
\begin{equation}
\label{eq:merom}
\varphi(z)=\sum_{l=1}^{n}\sum_{j=1}^{m_l}\frac{a_{lj}}{z-z_{lj}}+\tilde{\varphi}(z)
\end{equation}
for some $a_{lj}, z_{lj} \in\CC\setminus\{0\}$ and
$\tilde{\varphi}(z)\in\OO^{2}(\CC),$ where $j=1,\dots,m_l$, $l=1,\dots,n$ and $z_{l_j}$ are poles such
that $\delta_l:=2\*\arg (z_{l_1})=\dots=2\*\mathrm{arg}(z_{l_{m_l}})\in [0,2\pi)$ and $0\leq \delta_1<\delta_2<\dots<\delta_n< 2\pi$.
Set $\delta_{l+n}:=\delta_l+2\pi$ for $l=1,\dots,n$ and $u_k(t,z):= u^{\theta}(t,z)$ for $\theta\in(\delta_k,\delta_{k+1})\mod4\pi$
and $k=1,\dots,2n$ with $\delta_{2n+1}:=\delta_1+4\pi$,
where $u^{\theta}(t,z)$ is a solution of (\ref{eq:heat}) given by (\ref{eq:heat_solution}). Finally for $r>0$ denote $W_r=\{t\in W\colon 0<|t|<r\}$,
where $W$ is the Riemann surface of the square root function

Then for every $\varepsilon>0$ there exists $r>0$ such that $u_k(t,z)\in\OO(V_k(\varepsilon,r)\times D)$ (for $k=1,\dots,2n$),
where $V_k(\varepsilon,r):=\{t\in W\colon |t|\in(0,r), \arg t \in
(\delta_k-\frac{\pi}{2}+\frac{\varepsilon}{2},\delta_{k+1}+\frac{\pi}{2}-\frac{\varepsilon}{2}) \mod 4\pi\}$,
and 
$\{u_1,u_2,\dots,u_{2n}\}$ is a maximal family of solutions of (\ref{eq:heat}) on $W_r\times D$.
 Moreover, if $\hat{u}$ is a formal solution of (\ref{eq:heat}) then Stokes lines for $\hat{u}$ are given by the sets
 $\mathcal{L}_{\delta_k}$  and anti-Stokes lines by $\mathcal{L}_{\delta_k\pm\frac{\pi}{2}}$ for $k=1,\dots,2n$.
Jumps across the Stokes lines $\mathcal{L}_{\delta_l}$ and $\mathcal{L}_{\delta_{l+n}}$ ($l=1,\dots,n$) are equal respectively to
\begin{itemize}
 \item $u_l(t,z)-u_{l-1}(t,z)=u^{\delta_l^+}(t,z)-u^{\delta_l^-}(t,z)=-i\*\sqrt{\pi/t}\*\sum_{j=1}^{m_l}a_{lj}\,\*e^{-\frac{(z_{lj}-z)^2}{4t}}$
 for $(t,z)\in S(\delta_l,\pi-\varepsilon,r)\times D$ with the notation $u_0(t,z):=u_{2n}(t,z)$,
 \item $u_{l+n}(t,z)-u_{l-1+n}(t,z)=u^{\delta_{l+n}^+}(t,z)-u^{\delta_{l+n}^-}(t,z)
 =i\*\sqrt{\pi/t}\*\sum_{j=1}^{m_l}a_{lj}\,\*e^{-\frac{(z_{lj}-z)^2}{4t}}$ for $(t,z)\in S(\delta_{l+n},\pi-\varepsilon,r)\times D$.
\end{itemize}
\end{Wn}
\begin{proof} It is sufficient to calculate jumps across the Stokes lines $\mathcal{L}_{\delta_l}$ and $\mathcal{L}_{\delta_{l+n}}$.
To this end fix $l\in\{1,\dots,n\}$. For every $\varepsilon>0$ there exist $r,\tilde{r}>0$ such that the difference
$u^{\delta_l^+}(t,z)-u^{\delta_l^-}(t,z)$ satisfies (\ref{eq:difference}) for $(t,z)\in S(\delta_l,\pi-\varepsilon,r)\times D_{\tilde{r}}$ with
$B(G,\varphi,z)=\{z_{l_1}-z,\dots,z_{l_{m_l}}-z\}$. Hence
\begin{multline*}
u^{\delta_l^+}(t,z)-u^{\delta_l^-}(t,z)=-2\pi i\sum_{j=1}^{m_l}\mathrm{res}_{\tilde{s}=z_{l_j}-z}
 \Bigl[\frac{1}{\sqrt{4\pi\*t}}\bigl(\varphi(z+\tilde{s})+\varphi(z-\tilde{s})\bigr)\*e^{-{\tilde{s}}^2/4t}\Bigr]\\
=-i\*\sqrt{\pi/t}\*\sum_{j=1}^{m_l}a_{lj}\,\*e^{-\frac{(z_{lj}-z)^2}{4t}}.
\end{multline*}

To calculate the difference $u^{\delta_{l+n}^+}(t,z)-u^{\delta_{l+n}^-}(t,z)$, we repeat our previous considerations with
$B(G,\varphi,z)=\{z-z_{l_1},\dots,z-z_{l_{m_l}}\}$ for $(t,z)\in S(\delta_{l+n},\pi-\varepsilon,r)\times D_{\tilde{r}}$.
\end{proof}

\begin{Uw}
 As in Remark \ref{re:minus}, we can reduce the general Cauchy datum (\ref{eq:merom}) with $\arg z_{lj}\in[0,2\pi)$ to the case
 $\arg z_{lj}\in[0,\pi)$ assumed in Corollary \ref{co:merom}, by the eventual replacement $a_{lj}$ by $-a_{lj}$.
\end{Uw}

\begin{Wn}
\label{co:merom2}
Assume that the Cauchy datum of (\ref{eq:heat}) is given by
$$\varphi(z)=\frac{a_{-1}}{z-z_0}+\dots+\frac{a_{1-n}}{(z-z_0)^{n-1}}+\frac{a_{-n}}{(z-z_0)^n}+\tilde{\varphi}(z)$$
for some $ n\in\NN$,
$a_{-1},\dots,a_{1-n}\in\CC$, $a_{-n}, z_0\in\CC\setminus\{0\}$ and $\tilde{\varphi}(z)\in\OO^2(\CC)$.
In this case we obtain the same assertion as in Theorem \ref{th:main} but jumps across the Stokes lines $\mathcal{L}_{\delta}$
and $\mathcal{L}_{\delta+2\pi}$ are given respectively by
\begin{itemize}
 \item $u_1(t,z)-u_2(t,z)=u^{\delta^+}(t,z)-u^{\delta^-}(t,z)$ 
 $$={-i\sqrt{\pi/t}}\*\,\,\sum_{k=0}^{n-1}\Bigl[\frac{a_{k-n}}{(n-k-1)!}\*\lim_{s\longrightarrow z_0-z}\frac{d^{n-k-1}}{ds^{n-k-1}}
 \bigl(e^{\frac{-s^2}{4t}}\bigr)\Bigr] $$ for $(t,z)\in S(\delta,\pi-\varepsilon,r)\times D$,
\item $u_2(t,z)-u_1(t,z)=u^{(\delta + 2\pi)^+}(t,z)-u^{(\delta +2\pi)^-}(t,z)$
$$={-i\sqrt{\pi/t}}\*\,\,\sum_{k=0}^{n-1}\Bigl[\frac{(-1)^{k-n}a_{k-n}}{(n-k-1)!}\*\lim_{\tilde s\longrightarrow z-z_0}\frac{d^{n-k-1}}{d\tilde s^{n-k-1}}\bigl(e^{\frac{-\tilde s^2}{4t}}\bigr)
\Bigr], $$ for $(t,z)\in S(\delta+2\pi,\pi-\varepsilon,r)\times D.$
\end{itemize}
\end{Wn}
\begin{proof}
As in the proof of Theorem \ref{th:main} we have
\begin{multline*}
u_1(t,z)-u_2(t,z)=u^{\delta^+}(t,z)-u^{\delta^-}(t,z)\\
=-2\pi i\,\*\mathrm{res}_{\tilde{s}=z_0-z}\,\Bigl[\frac{1}{\sqrt{4\pi\*t}}
\bigl(\varphi(z+\tilde{s})+\varphi(z-\tilde{s})\bigr)\*e^{-{\tilde{s}}^2/4t}\Bigr]\\ 
=\frac{-i\sqrt{\pi/t}}{(n-1)!}\lim_{\tilde s\longrightarrow z_0-z}\frac{d^{n-1}}{d\tilde s^{n-1}}
\Bigl[\Bigl(\frac{a_{-1}}{(\tilde s-(z_0-z))^{1-n}}+\dots+\frac{a_{1-n}}{(\tilde s-(z_0-z))^{-1}}\\
+\,a_{-n}+(\tilde s-(z_0-z))^n\*\,\tilde{\varphi}(z+\tilde s)\Bigr)\*e^{-\tilde s^2/4t}\Bigr]\\
={-i\sqrt{\pi/t}}\*\,\,\sum_{k=0}^{n-1}\Bigl[\frac{a_{k-n}}{(n-k-1)!}\*\lim_{\tilde s\longrightarrow z_0-z}\frac{d^{n-k-1}}{d\tilde s^{n-k-1}}
\bigl(e^{\frac{-\tilde s^2}{4t}}\bigr)\Bigr]
\end{multline*}
for $(t,z)\in S(\delta,\pi-\varepsilon,r)\times D$.

Analogously
\begin{multline*}
u_2(t,z)-u_1(t,z)=u^{(\delta + 2\pi)^+}(t,z)-u^{(\delta +2\pi)^-}(t,z)\\
=-2\pi i\,\*\mathrm{res}_{\tilde{s}=z-z_0}\,\Bigl[\frac{1}{\sqrt{4\pi\*t}}
\bigl(\varphi(z+\tilde{s})+\varphi(z-\tilde{s})\bigr)\*e^{-{\tilde{s}}^2/4t}\Bigr]\\ 
=\frac{-i\sqrt{\pi/t}}{(n-1)!}\lim_{\tilde s\longrightarrow z-z_0}\frac{d^{n-1}}{d\tilde s^{n-1}}
\Bigl[\Bigl(\frac{-a_{-1}}{(\tilde s-(z-z_0))^{1-n}}+\dots+\frac{(-1)^{1-n} a_{1-n}}{(\tilde s-(z-z_0))^{-1}}\\
+\,(-1)^{-n}a_{-n}+(\tilde s-(z-z_0))^n\*\,\tilde{\varphi}(z-\tilde s)\Bigr)\*e^{-\tilde s^2/4t}\Bigr]\\
={-i\sqrt{\pi/t}}\*\,\,\sum_{k=0}^{n-1}\Bigl[\frac{(-1)^{k-n}a_{k-n}}{(n-k-1)!}\*\lim_{\tilde s\longrightarrow z-z_0}
\frac{d^{n-k-1}}{d\tilde s^{n-k-1}}\bigl(e^{\frac{-\tilde s^2}{4t}}\bigr)\Bigr] 
\end{multline*}
for $(t,z)\in S(\delta+2\pi,\pi-\varepsilon,r)\times D$.
\end{proof}

\begin{Uw}
Analogously to Corollaries \ref{co:merom} and \ref{co:merom2} we can also derive the Stokes lines, anti-Stokes lines and jumps in the general case,
when the Cauchy datum $\varphi$ of (\ref{eq:heat}) is given by a meromorphic function with finite number of poles, i.e.
\begin{gather}
\label{eq:general_cauchy}
\varphi(z)=\sum_{l=1}^{n}\sum_{k=1}^{r_l}\frac{a_{lk}}{(z-z_l)^k}+\tilde{\varphi}(z)\quad\textrm{with}\quad z_l\in\CC\setminus\{0\},\quad \tilde{\varphi}(z)\in\OO^2(\CC).
\end{gather}
\end{Uw}

\section {Generalizations of the heat equation}
In this section we will generalize the results presented in the previous part. 

Let us consider the equation  
\begin{equation}
 \label{eq:generalization}
 \begin{cases}
\partial_t^p u(t,z)=\partial_z^q u(t,z),\,\, p,q\in\NN \\
u(0,z)=\varphi(z) \in\OO(D)\\
\partial_{t}^{j}u(0,z)=0\ \textrm{for}\,\, j=1,2,\dots,p-1.
 \end{cases}
\end{equation}

The equation above has a unique formal solution represented by $$\hat{u}(t,z)=\sum_{n=0}^{\infty}\frac{\varphi^{(qn)}(z)\* t^{pn}}{(pn)!}.$$

The result for the summability of $\hat{u}(t,z)$ was proved by M.~Miyake \cite{Miy}.
An integral representation of the $\frac{p}{q-p}$-sum $u(t,z)$, in terms of the Barnes generalized hypergeometric series,
was obtained by K.~Ichinobe \cite[Theorem 5.1]{I2}.

Now we present another form of $\frac{p}{q-p}$-sum which appears to be more useful to study the Stokes phenomenon.

\begin{Tw}
\label{th:generalization}
Let $d\in\RR$. Suppose that $\hat{u}(t,z)$ is a unique formal solution of the Cauchy problem (\ref{eq:generalization}) with $1\leq p<q$ and
\begin{gather}
\label{eq:generalization_cond}
\varphi(z)\in\OO^{\frac{q}{q-p}}\Bigl(D\cup \bigcup_{l=0}^{q-1} S(\frac{dp}{q}+\frac{2\pi\*l}{q},\frac{\varepsilon p}{q})\Bigr)\quad
\textit{for some}\quad \varepsilon>0.
\end{gather}

Then $\hat{u}(t,z)$ is $\frac{p}{q-p}$-summable  in the direction $d$ and for every $\theta\in(d-\frac{{\varepsilon}}{2},d+\frac{{\varepsilon}}{2})$
and
for every $\tilde{\varepsilon}\in(0,\varepsilon)$ there exists $r>0$ such that its $\frac{p}{q-p}$-sum
$u^{\theta}\in\OO(S(\theta,\frac{\pi(q-p)}{p}-\tilde{\varepsilon},r)\times D)$ is given by
\begin{multline}
 \label{eq:generalization_solution}
u(t,z)=u^{\theta}(t,z)=E_{\frac{p}{q-p},\theta}\breve{B}_{\frac{p}{q-p}}\hat{u}(t,z)\\
=\frac{1}{q\*\sqrt[q]{t^p}}\*\int_{e^{\frac{i\theta p}{q}}\RR_+}\bigl(\varphi(z+\tilde{s})+\dots+\varphi(z+e^{\frac{2(q-1)\pi\*i}{q}}\*\tilde{s})\bigr)\,C_{\frac{q}{p}}(\tilde{s}/\sqrt[q]{t^p})\,d\tilde{s}
\end{multline}
for $t\in S(\theta,\frac{\pi(q-p)}{p}-\tilde{\varepsilon},r)$ and $z\in D$.

Moreover, for every $\bar{\varepsilon}\in(0,\varepsilon)$ there exists
$r>0$ such that $u\in\OO(S(d,\frac{\pi(q-p)}{p}+\bar{\varepsilon},r)\times D)$.
\end{Tw}

\begin{proof}
Using the procedure introduced in Section \ref{sect}, firstly we have
\begin{multline*}
v(t,z)=(\breve{B}_{\frac{p}{q-p}}\hat{u})(t,z)=\sum_{n=0}^{\infty}\frac{\varphi^{(qn)}(z)\*(t^{\frac{p}{q}})^{qn}}{(qn)!}\\
=\frac{1}{q}\*\bigl(\varphi(z+\sqrt[q]{t^p})+\varphi(z+e^{\frac{2\pi\*i}{q}}\*\sqrt[q]{t^{p}})+\dots+\varphi(z+e^{\frac{2(q-1)\pi\*i}{q}}\*\sqrt[q]{t^p})\bigr).
\end{multline*}
Since $\varphi(z)\in\OO^{\frac{q}{q-p}}\Bigl(D\cup\,\, \bigcup_{l=0}^{q-1} S(\frac{dp}{q}+\frac{2\pi\*l}{q},\frac{\varepsilon p}{q})\Bigr)$ for some
$\varepsilon>0$, we get $v(t,z)\in\OO^{\frac{p}{q-p}}\bigl((S(d,\varepsilon)\cup D)\times D\bigr).$

We obtain
\begin{multline*}
u^{\theta}(t,z)=(E_{\frac{p}{q-p},\theta}\,v)(t,z)=
\frac{1}{\sqrt[q]{t^p}}\*\int_{e^{i\*\theta}\RR_+}v(s,z)\,C_{\frac{q}{p}}((s/t)^{\frac{p}{q}})\,ds^{\frac{p}{q}}\\
=\frac{1}{\sqrt[q]{t^p}}\*\int_{e^{i\*\theta}\RR_+}\frac{1}{q}\*\bigl(\varphi(z+\sqrt[q]{s^p})+\dots+\varphi(z+e^{\frac{2(q-1)\pi\*i}{q}}\*\sqrt[q]{s^p})\bigr)
C_{\frac{q}{p}}((s/t)^{\frac{p}{q}})ds^{\frac{p}{q}}\\
\stackrel{\sqrt[q]{s^p}= \tilde s}{=} \frac{1}{q\*\sqrt[q]{t^p}}\*
\int_{e^{\frac{i\theta p}{q}}\RR_+}\bigl(\varphi(z+\tilde{s})+\dots+\varphi(z+e^{\frac{2(q-1)\pi\*i}{q}}\*\tilde{s})\bigr)C_{\frac{q}{p}}(\tilde{s}/\sqrt[q]{t^p})d\tilde{s},
\end{multline*}
for $\arg t\in(-\frac{\pi\*(q-p)}{2p}+\tilde{\varepsilon}+\theta, \frac{\pi\*(q-p)}{2p}-\tilde{\varepsilon}+\theta)$.

Taking $k=\frac{p}{q-p}$ in Remark \ref{re:rem 3} we conclude that for every $\bar{\varepsilon}\in(0,\varepsilon)$ there exists
$r>0$ such that $u\in\OO(S(d,\frac{\pi(q-p)}{p}+\bar{\varepsilon},r)\times D)$.
\end{proof}

\begin{Tw}
Assume that the Cauchy datum of (\ref{eq:generalization}) is given by
\begin{gather}
\label{eq:heat_gen}
\varphi(z)=\frac{a}{z-z_0}+\tilde{\varphi}(z)\quad \text{for some} \quad a,z_0 \in\CC\setminus\{0\} \quad \text{and}\quad
\tilde{\varphi}(z)\in\OO^{\frac{q}{q-p}}(\CC). 
\end{gather}
Set $\delta_l:=\frac{q}{p}\arg z_0+\frac{2(l-1)\pi}{p}$ ($l=1,\dots,q+1$) and $u_l:= u^{\theta}$ for
$\theta\in(\delta_l,\delta_{l+1})\mod\frac{2q\pi}{p}$ and $l=1,\dots,q$, where 
$u^{\theta}$ is a solution of (\ref{eq:generalization}) given by 
(\ref{eq:generalization_solution}). Finally for $r>0$ denote
 $W_r=\{t\in W\colon 0<|t|<r\}$, where $W$ is the Riemann surface of the function $t\mapsto t^{\frac{p}{q}}$.

Then for every $\varepsilon>0$ there exists $r>0$ such that
$u_l\in\OO(V_l\times D)$ ($l=1,\dots,q$), where $V_l=S(\delta_l, \frac{2\pi}{p}+\frac{\pi(q-p)}{p}-\varepsilon,r)$,
and $\{u_1,u_2,\dots, u_q\}$ is a maximal family of solutions of (\ref{eq:generalization}) on $W_r\times D$.
 
Moreover, if $\hat{u}$ is a formal solution of (\ref{eq:generalization}) then Stokes lines for $\hat{u}$ are the sets
$\mathcal{L}_{\delta_l}$ and anti-Stokes lines for $\hat{u}$ are the sets
$\mathcal{L}_{\delta_l\pm\frac{\pi(q-p)}{2p}}$.
Jumps across the Stokes lines $\mathcal{L}_{\delta_l}$ ($l=1,\dots,q$) are given by
$$
u_{l}(t,z)-u_{l-1}(t,z)=u^{\delta_l^+}(t,z)-u^{\delta_l^-}(t,z)=
-\frac{2\pi\*i}{q\*\sqrt[q]{t^p}}\,\*e^{-\frac{2(l-1)\pi i}{q}}\*a\,\*C_{\frac{q}{p}}(\frac{(z_0-z)e^{-\frac{2(l-1)\pi i}{q}}}{\sqrt[q]{t^p}})
$$
for $(t,z)\in S(\delta_l,\frac{\pi(q-p)}{p}-\varepsilon,r)\times D$ with the notation $u_0(t,z)=u_q(t,z)$.
\end{Tw}

\begin{proof}
Let $\hat{u}$ be a formal solution of (\ref{eq:generalization}) with the Cauchy datum $\varphi$ given by (\ref{eq:heat_gen}).
Observe that, if $d\neq\delta_l\mod \frac{2q\pi}{p}$ for $l=1,\dots,q$ then $\varphi$ satisfies assumption (\ref{eq:generalization_cond}).

Hence by Theorem \ref{th:generalization}, $\hat{u}$ is $\frac{p}{q-p}$-summable in a direction $\theta\in\RR$, $\theta\neq \delta_l\mod \frac{2q\pi}{p}$  for $l=1,\dots,q$
and for every $\varepsilon>0$ there exists $r>0$ such that its $\frac{p}{q-p}$-sum satisfies
$$u^\theta (t,z)=\frac{1}{q\*\sqrt[q]{t^p}}\*\int_{e^{\frac{i\theta p}{q}}\RR_+}\bigl(\varphi(z+\tilde{s})+\dots+\varphi(z+e^{\frac{2(q-1)\pi\*i}{q}}\*\tilde{s})\bigr)\,C_{\frac{q}{p}}(\tilde{s}/\sqrt[q]{t^p})\,d\tilde{s}$$
for $(t,z)\in S(\theta,\frac{\pi(q-p)}{p}-\varepsilon,r)\times D$.

Observe that $u^{\theta}(t,z)=u^{\theta}(te^{\frac{2q\pi i}{p}},z)$ and $q$ is the smallest positive rational number for which this equality holds.
Moreover, the set of singular directions of $\hat{u}(t,z)$ modulo $\frac{2q\pi}{p}$ is given by $\delta_l\mod \frac{2q\pi}{p}$ for $l=1,\dots,q$.
Hence by Theorem \ref{th:general},
$\{u_1,u_2\dots,u_q\}$ with $u_l\in\OO(V_l\times D)$ ($l=1,\dots,q$) is a maximal family of solutions of (\ref{eq:generalization}) . Moreover,
Stokes lines for $\hat{u}$ are the sets
$\mathcal{L}_{\delta_l}$ and anti-Stokes lines for $\hat{u}$ are sets
$\mathcal{L}_{\delta_l\pm\frac{\pi(q-p)}{2p}}$.

To calculate jumps across the Stokes lines, take $\theta_1,\theta_2 \neq \delta_l\mod \frac{2q\pi}{p} $ and
suppose $0<\theta_2-\theta_1<\frac{2\pi}{p}$.
Proceeding as in the proof of Theorem \ref{th:main}, we obtain
\begin{multline}
 \label{eq:differ}
 u^{\theta_2}(t,z)-u^{\theta_1}(t,z)\\
 =-2\pi\*i\*\sum_{\tilde z\in B_{p,q}(G,\varphi,z)}\mathrm{res}_{\tilde s=\tilde z}\,
 \Bigl[\frac{1}{q\*\sqrt[q]{t^p}}\bigl(\varphi(z+\tilde{s})+\dots+\varphi(z+e^{\frac{2(q-1)\pi\*i}{q}}\*\tilde{s})\bigr)\,C_{\frac{q}{p}}(\tilde{s}/\sqrt[q]{t^p})\Bigr]
\end{multline}
for $t\in S(\theta_1,\frac{\pi(q-p)}{p}-\varepsilon,r)\cap S(\theta_2,\frac{\pi(q-p)}{p}-\varepsilon,r)$ and $z\in D$,
 where $G=G(\frac{p\theta_1}{q},\frac{p\theta_2}{q})=\{z\in\tilde{\CC}\colon \arg z\in(\frac{p\theta_1}{q},\frac{p\theta_2}{q})\}$ and 
 $B_{p,q}(G,\varphi,z)$ is the set of all points in the sector $G$
at which the function
$$\tilde{s}\mapsto\frac{1}{q\*\sqrt[q]{t^p}}\bigl(\varphi(z+\tilde{s})+\dots+\varphi(z+e^{\frac{2(q-1)\pi\*i}{q}}\*\tilde{s})\bigr)\,C_{\frac{q}{p}}(\tilde{s}/\sqrt[q]{t^p})
$$
has poles.
\bigskip\par
Analogously to the proof of Theorem \ref{th:main} we consider the following cases:
\bigskip\par
Case 1: Let $\theta_1,\theta_2\in (\delta_{l-1},\delta_l)$ for some $l=1,\dots,q.$
Since $\theta_1,\theta_2\in (\delta_{l-1},\delta_l)$, there exists $ \tilde{r}>0$ such that
$ B(z_0e^{-\frac{2(j-1)\pi i}{q}},\tilde{r})\cap G=\emptyset$ for $j=1,\dots,q$. Then we see that for $z\in D_{\tilde{r}}$
the set $B_{p,q}(G,\varphi,z)$ is empty,
and hence by (\ref{eq:differ})
$$u^{\theta_2}(t,z)-u^{\theta_1}(t,z)=0.$$ 
\bigskip\par
Case 2: Let $\theta_1<\delta_l<\theta_2$ for some $l=1,\dots,q.$ Then there exists $\tilde r>0$ such that
$B(z_0e^{-\frac{2(l-1)\pi i}{q}},\tilde r)\subset G$. 
Hence we get $B_{p,q}(G,\varphi,z)=\{(z_0-z)e^{-\frac{2(l-1)\pi i}{q}}\}$ for $z\in D_{\tilde{r}}$. 
Consequently, by (\ref{eq:differ}), we obtain
\begin{multline*}
u^{\theta_2}(t,z)-u^{\theta_1}(t,z)\\
=
-2\pi\*i\,\*\mathrm{res}_{\tilde{s}=(z_0-z)e^{-\frac{2(l-1)\pi i}{q}}}\,\Bigl[\frac{1}{q\*\sqrt[q]{t^p}}\bigl(\varphi(z+\tilde{s})+\dots+
\varphi(z+e^{\frac{2(q-1)\pi\*i}{q}}\*\tilde{s})\bigr)\,C_{\frac{q}{p}}(\tilde{s}/\sqrt[q]{t^p})\Bigr]\\
=-\frac{2\pi i}{q\sqrt[q]{t^p}}\mathrm{res}_{\tilde{s}= (z_0-z)e^{-\frac{2(l-1)\pi i}{q}}}\Bigl[\bigl(\frac{a}{\tilde{s}-(z_0-z)}+\dots
+\frac{e^{-\frac{2(q-1)\pi\*i}{q}}a}{\tilde{s}-e^{-\frac{2(q-1)\pi\*i}{q}}(z_0-z)}\bigr)\,
C_{\frac{q}{p}}(\tilde{s}/\sqrt[q]{t^p})\Bigr]\\
=-\frac{2\pi\*i}{q\*\sqrt[q]{t^p}}\,\*e^{-\frac{2(l-1)\pi i}{q}}\*a\,\*C_{\frac{q}{p}}(\frac{(z_0-z)e^{-\frac{2(l-1)\pi i}{q}}}{\sqrt[q]{t^p}}),
\end{multline*}
for $t\in S(\theta_1,\frac{\pi(q-p)}{p}-\varepsilon,r)\cap S(\theta_2\frac{\pi(q-p)}{p}-\varepsilon,r)$ and $z\in D_{\tilde{r}}$.

If we put $\theta_1=\delta_l^-$
and $\theta_2=\delta_l^+$, then we have $(t,z)\in S(\delta_l,\frac{\pi(q-p)}{p}-\varepsilon,r)\times D_{\tilde{r}}$ for small $\varepsilon>0$.
So we get $$u^{\delta_l^+}(t,z)-u^{\delta_l^-}(t,z)=
-\frac{2\pi\*i}{q\*\sqrt[q]{t^p}}\,\*e^{-\frac{2(l-1)\pi i}{q}}\*a\,\*C_{\frac{q}{p}}(\frac{(z_0-z)e^{-\frac{2(l-1)\pi i}{q}}}{\sqrt[q]{t^p}}),
$$
for $(t,z)\in S(\delta_l,\frac{\pi(q-p)}{p}-\varepsilon,r)\times D_{\tilde r}.$

Notice that for such $(t,z)$ we have
$$-\frac{2\pi\*i}{q\*\sqrt[q]{t^p}}\,\*e^{-\frac{2(l-1)\pi i}{q}}\*a\,\*C_{\frac{q}{p}}(\frac{(z_0-z)e^{-\frac{2(l-1)\pi i}{q}}}{\sqrt[q]{t^p}})\sim_{\frac{p}{q-p}} 0.
$$
Therefore  $u^{\delta^+}(t,z)\sim_{\frac{p}{q-p}}\hat{u}(t,z)$ and $u^{\delta^-}(t,z)\sim_{\frac{p}{q-p}}\hat{u}(t,z)$.
\end{proof}

\begin{Uw}
Similarly to the heat equation we can also describe Stokes lines, anti-Stokes lines and jumps in case when the Cauchy datum $\varphi$ 
of (\ref{eq:generalization}) is given by a meromorphic function (\ref{eq:general_cauchy}) with $\tilde{\varphi}(z)\in\OO^{\frac{q}{q-p}}(\CC)$.
\end{Uw}

\section{Final remarks}
In the authors' opinion, using the methods of W.~Balser \cite{B5} and S.~Michalik \cite{Mic5}, one can extend the crucial results obtained
in this work to
general linear partial differential equations with constant coefficients
(\ref{eq:pde}) and with meromorphic Cauchy data.

Similarly, it seems to be possible to generalize the results of the paper to moment partial differential equations introduced
by W.~Balser and M.~Yoshino \cite{B-Y} and developed by S. Michalik \cite{Mic7,Mic8}.

One can also try to investigate the Stokes phenomenon for summable solutions of heat type equations on a real analytic manifold
using the results of G. {\L}ysik \cite{L4}.

We are going to study these problems in the subsequent papers.

\subsection*{Acknowledgements}
The authors would like to thank the anonymous referees for valuable comments and suggestions.

\bibliographystyle{siam}
\bibliography{summa}
\end{document}